\documentclass{article}

\usepackage{amsmath}
\usepackage{amssymb}
\usepackage{amsthm}
\usepackage{tikz}
\usepackage{verbatim}

\usepackage[T2A]{fontenc}
\usepackage[utf8]{inputenc}

\newtheorem{theorem}{Theorem}
\newtheorem{prop}[theorem]{Proposition}
\newtheorem{lemma}[theorem]{Lemma}
\newtheorem{cor}[theorem]{Corollary}
\newtheorem{conjecture}{Conjecture}
\newtheorem{definition}{Definition}

\newcommand{\R}{\mathbb{R}}
\newcommand{\Z}{\mathbb{Z}}

\DeclareMathOperator{\Sh}{\textup{Ш}}
\DeclareMathOperator{\ash}{\overline{\Sh}}
\DeclareMathOperator{\cat}{cat}
\DeclareMathOperator{\wt}{wt}

\newcommand{\half}{\frac{1}{2}}
\newcommand{\linegraph}[1]{L^{(#1)}}

\title{Graham's Tree Reconstruction Conjecture and a Waring-Type Problem on Partitions}%\thanksref{t1}}
\author{Joshua Cooper, Bill Kay, and Anton Swifton}

\begin{document}

\maketitle

\begin{abstract}
Suppose $G$ is a tree.  Graham's ``Tree Reconstruction Conjecture'' states that $G$ is uniquely determined by the integer sequence $|G|$, $|L(G)|$, $|L(L(G))|$, $|L(L(L(G)))|$, $\ldots$, where $L(H)$ denotes the line graph of the graph $H$.  Little is known about this question apart from a few simple observations.  We show that the number of trees on $n$ vertices which can be distinguished by their associated integer sequences is $e^{\Omega((\log n)^{3/2})}$.  The proof strategy involves constructing a large collection of caterpillar graphs using partitions arising from the Prouhet-Tarry-Escott problem.
\end{abstract}

\section{Introduction}

A conjecture of R.~L.~Graham (see, e.g., \cite{GR01}), often referred to as the ``Tree Reconstruction Conjecture'', states that, if $G$ is a tree, then $G$ is uniquely determined by the sequence of sizes of its iterated line graphs.  To make this statement precise, we start with a few definitions.  All graphs are taken to be simple and undirected; a {\it tree} is an acyclic, connected graph.  Given a graph $G = (V,E)$, define the {\em line graph} $L(G)$ to be a graph with vertex set $E$, and for distinct $e,f \in E$ we have $\{e,f\} \in E(L(G))$ iff $e \cap f \neq \emptyset$, i.e., $e$ and $f$ are incident in $G$.  We denote the $j^{\textrm{th}}$-iterated line graph by $L^{(j)}(G)$. $L^{(0)}(G) = G$ and $L^{(j+1)}(G) = L(L^{(j)}(G))$ for $j \geq 0$.

\begin{definition}
The Graham sequence of a graph \(G\) is the sequence of sizes of its iterated line graphs \(|\linegraph{0}(G)|, |\linegraph{1}(G)|, |\linegraph{2}(G)|, \ldots\)
\end{definition}

\begin{conjecture}[Graham]
\label{Graham} For each sequence of natural numbers $a_0,a_1,a_2,\ldots$, all the conditions $|L^{(j)}(G)| = a_j$ for $j \geq 0$ are satisfied by at most one tree $G$.
\end{conjecture}

If $G$ and $H$ are two graphs, we say that they are {\em Graham equivalent} if \newline $|L^{(j)}(G)| = |L^{(j)}(H)|$ for all $j \geq 0$. The corresponding equivalence classes we call {\em Graham classes}.  We can reformulate Conjecture \ref{Graham} as follows:

\begin{conjecture}
For each $n \geq 1$, the number of Graham classes of trees on $n$ vertices equals the number of isomorphism classes of trees on $n$ vertices.
\end{conjecture}

As shown by Otter (\cite{O48}), the number of isomorphism classes of trees on $n$ vertices is $\widetilde{\Theta}(\alpha^{n})$, where $\alpha = 2.955765\ldots$, i.e., approximately $3^n$.  Our main result is a lower bound on the number of Graham classes of trees that is superpolynomial, although substantially subexponential.

\begin{theorem} \label{thm:mainthm} The number of Graham classes of trees on $n$ vertices is
$$
e^{\Omega( (\log n)^{3/2})}
$$
\end{theorem}

In order to describe the method of proof, we need a few (mostly standard) definitions. A {\it path} of length $n$, denoted $P_n$, is a tree on the vertex set $\{v_0,\ldots,v_n\}$ with an edge between $v_j$ and $v_{j+1}$ for each $j$, $0 \leq j < n$.  A {\it pendant vertex} in a graph $G$ is a vertex of degree one.  A {\it caterpillar} is a graph obtained from a path by attaching pendant vertices to some of the path vertices.  The path from which a caterpillar is built is its {\it spine}, the vertices on the path of degree greater than two are {\it joints}, and the pendant vertices attached to the path are {\it legs}.

The proof proceeds as follows.  We construct a collection of caterpillars $\{G_j\}$ on $n$ vertices with distinct Graham sequences.  To ensure that the Graham sequences differ, we choose the degrees $d_1$, $\ldots$, $d_t$ of specially selected joints to be a particular class of partitions associated with the Prouhet-Tarry-Escott problem, and leave the rest of the vertices legless.  We show that for each $k$ there exists a degree $k$ polynomial $f_k$ such that, for some constant $C_{n,k,t}$ depending on $n$, $k$, and $t$, 
\begin{equation} \label{lgtopolyn}
|L^{(k)}(G_j)| = C_{n,k,t} + \sum_{i = 1}^t f_k(d_i), 
\end{equation}
where $\{d_i\}$ is the degree sequence of the joints of $G_j$.

We will also need to bound from above the ratio of the largest coefficient in the relevant polynomial to its lead coefficient.  Much of the work consists of obtaining such bounds; it should be noted, however, that we make little attempt to optimize the resulting expressions other than to simplify exposition.

Finally, we construct a sufficient number of partitions \((d_1, \ldots, d_t)\) such that caterpillars constructed in correspondence to these partitions have the same number of vertices, while their Graham sequences are different.

\section{From Caterpillars to Polynomials}

Given a sequence of positive integers $d = (d_1, \ldots, d_t)$ and $m > 0$ define $\cat(d_1,\ldots, d_t; m)$ to be a caterpillar graph
whose spine is a path of length $(t+1)m-2$ on the vertex set $v_1,...,v_{m(t+1)-1}$, with $d_i$ legs attached to vertex $v_{im}$ for $1 \leq i \leq t$. We call $d$ the {\em joint degree sequence}  of $\cat(d_1,\ldots, d_t; m)$.  We will eventually define the aforementioned $G_i$ as a modified $\cat(d_1,\ldots , d_t;m)$ with suitably chosen parameters. Write $S(d;a,b)$ for a star with ``central vertex'' of degree $d$ to which two disjoint paths are appended at their endvertices: one of length $a$ and one of length $b$.  (See Figure 1.)

\begin{figure}[h]
\begin{center}
 \begin{tikzpicture}
    \tikzstyle{every node}=[draw,circle,fill=white,minimum size=4pt,
                            inner sep=0pt]

    \node (CTR) at (0,0) {};

    \draw (CTR) -- ++(18:1cm) node (LEG1) {};
    \draw (CTR) -- ++(69:1cm) node (LEG1) {};
    \draw (CTR) -- ++(121:1cm) node (LEG1) {};
    \draw (CTR) -- ++(224:1cm) node (LEG1) {};
    \draw (CTR) -- ++(275:1cm) node (LEG1) {};

    \draw (CTR) -- ++(172:1cm) node (PATH1A) {};
    \draw (PATH1A) -- ++(160:1cm) node (PATH1B) {};
    \draw (PATH1B) -- ++(180:1cm) node (PATH1C) {};

    \draw (CTR) -- ++(326:1cm) node (PATH2A) {};
    \draw (PATH2A) -- ++(356:1cm) node (PATH2B) {};
    \draw (PATH2B) -- ++(386:1cm) node (PATH2C) {};
    \draw (PATH2C) -- ++(356:1cm) node (PATH2D) {};
    \draw (PATH2D) -- ++(386:1cm) node (PATH2E) {};
    \draw (PATH2E) -- ++(416:1cm) node (PATH2F) {};
    \draw (PATH2F) -- ++(376:1cm) node (PATH2G) {};
\end{tikzpicture}
\end{center}
\caption{An $S(5;3,7)$. Alternatively, a $\cat(5,0;4)$.}
\end{figure}
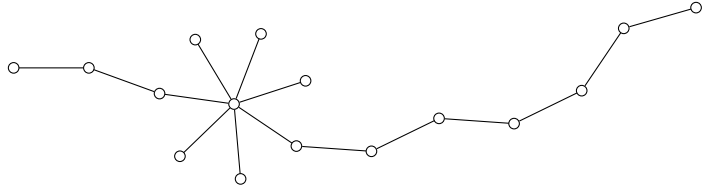

\begin{definition}
For \(X \subseteq V(G)\) define the \(i\)-th antishadow of \(X\) as \[\ash_i (X) = V(\linegraph{i} (G)) \setminus V(\linegraph{i} (G-X))\] 
\end{definition}
Intuitively, antishadow is the set of  vertices in the \(i^{th}\) line graph affected by the vertices in \(X\) and their edges. The following propositions regarding antishadows will allow us to break down the \(k^{th}\) line graph of a caterpillar $\cat(d_1,\ldots, d_t; m)$ into a union of line graphs of stars. 

\begin{prop}\label{edge}
Let \(v = \{w_1, w_2\} \in V(\linegraph{i + 1} (G))\) and \(X \subseteq V(G)\). \newline Then \(v \in \ash_{i + 1} (X)\) iff either \(w_1 \in \ash_i(X)\) or \(w_2 \in \ash_i(X)\) (or both).
\end{prop}
\begin{proof}
\(v \in \ash_{i+1} (X) = V(\linegraph{i+1} (G)) \setminus V(\linegraph{i+1} (G-X))\) if and only if \(v \in V(\linegraph{i+1} (G)) = E(\linegraph{i} (G)\) and \(v \not \in V(\linegraph{i+1} (G - X)) =  E(\linegraph{i} (G - X)).\) Since \(v = \{w_1, w_2\}\), this is equivalent to \(w_1, w_2 \in V(\linegraph{i} (G))\) and either \(w_1 \not \in V(\linegraph{i}(G-X))\) or \(w_2 \not \in V(\linegraph{i}(G-X))\). Equivalently, either \(w_1 \in \ash_i(X)\) or \(w_2 \in \ash_i(X)\).
\end{proof}

\begin{prop}\label{pathsinilg}
If \(u, v \in \linegraph{i+1} (G)\) are connected by a path of length at most \(q\), and \(u = \{u_1, u_2\}, v = \{v_1, v_2\}\), then \(u_p\) and \(v_s\) are connected by a path in \(\linegraph{i} (G)\) of length at most \(q+1\) for \(p, s = 1, 2\). 
\end{prop}
\begin{proof}
Let \(u = w_0, \ldots, w_n = v\) be the shortest path connecting \(u\) and \(v\). Due to the assumption of the proposition, \(n \leq q\). Since for any \(j = 1, \ldots, n, ~ w_{j - 1}\) and \(w_j\) are connected by an edge, the correspoding edges in \(\linegraph{i} (G)\) have a common vertex: \(w_{j - 1} \cap w_j = t_{j - 1}\). All vertices \(t_j\) are different, otherwise, if \(t_i = t_j, j > i\), then the edges \(w_i\) and \(w_{j+1}\) are incident, and the original path is not the shortest one. Since all \(t_j\) are different, \(t_0, u, t_1, w_1, t_2, \ldots, t_n, w_n, t_{n+1}\) is a path in \(\linegraph{i}(G)\). This implies that \(d(t_0, t_{n+1}) \leq n + 1 \leq q + 1\). Since \(u = \{t_0, t_1\} = \{u_1, u_2\}\) and \(v = w_n = \{t_{n-1}, t_n\} = \{v_1, v_2\}\), we have produced a path from $u_p$ to $v_s$ of length at most $q+1$ for \(p, s = 1,2\).
\end{proof}

\begin{cor}
If \(X, Y \subseteq V(G)\) and \(d(\ash_i (X), \ash_i (Y)) \geq q + 1\), then \newline \(d(\ash_{i+1} (X), \ash_{i+1} (Y)) \geq q\).
\end{cor}

\begin{cor}
If \(d(X, Y) > m\), then \(\ash_m (X) \cap \ash_m (Y) = \emptyset\).
\end{cor}

\begin{prop}
If \(\ash_m(X) \cap \ash_m (Y) = \emptyset\), then \(\ash_m (X \cup Y) = \ash_m (X) \cup \ash_m(Y)\).
\end{prop}
\begin{proof}

We proceed by induction. The base case  is trivial. Assume \(\ash_{m-1}(X \cup Y) = \ash_{m-1} (X) \cup \ash_{m-1}(Y)\) and \(\ash_m(X) \cap \ash_m (Y) = \emptyset\). Then \(v = \{w, u\} \in \ash_m (X \cup Y)\) iff (without loss of generality) \(w \in \ash_{m - 1} (X \cup Y) = \ash_{m-1}(X) \cup \ash_{m-1}(Y)  \Leftrightarrow\) either \(w \in \ash_{m-1} (X)\) or \(w \in \ash_{m-1} (Y) \Leftrightarrow\) either \(v \in \ash_m (X)\) or \(v \in \ash_m (Y) \Leftrightarrow v \in \ash_m(X) \cup \ash_m(Y)\). 
\end{proof}

Now we can compute the number of vertices in the iterated line graph of a caterpillar by considering  simple pieces. 

\begin{lemma}\label{cattostars}
Let \(m > k\). Then \[|L^{(k)}(\cat(d_1,\ldots,d_t;m))| = (t-1)(k-m) +  \sum_{j=1}^t |L^{(k)}(S(d_j;m,m))|\].
\end{lemma}
\begin{proof}
Let \(D_1, \ldots, D_t\) be the sets of pendant vertices, where each  $D_i$ is a maximal set of pendant vertices  attached to the same spine vertex. Then \(|D_j| = d_j, j = 1, \ldots, t\), and let \(D = \cup_{j=1}^t D_j\). If \(G =\cat(d_1, \ldots, d_t)\), then 
\begin{align*}
V(\linegraph{k} (G)) &= V(\linegraph{k} (G)) \setminus V(\linegraph{k} (G - D)) \cup V(\linegraph{k} (G - D))\\
& = \ash_k (D) \cup V(\linegraph{k} (G - D)) \\
& = \cup_{j = 0}^t \ash_k(D_j) \cup V(\linegraph{k} (P_{(t+1)m - 1}))
\end{align*}
Also, \(\ash_k (D_j) = V( \linegraph{k} (S(d_j, m, m))) \setminus V( \linegraph{k} (P_{2m})\), and therefore the following holds: 
\begin{align*}
|L^{(k)}(\cat(d_1,\ldots,d_t;m))| &= |L^{(k)}(P_{(t+1)m-1})| \\
& \qquad + \sum_{j=1}^t \left ( |L^{(k)}(S(d_j;m,m))| - |L^{(k)}(P_{2m})| \right )\\
&= (t+1)m - 1 - k - t (2m-k) \\ 
& \qquad + \sum_{j=1}^t |L^{(k)}(S(d_j;m,m))|\\
&= (t-1)(k-m) - 1 + \sum_{j=1}^t |L^{(k)}(S(d_j;m,m))|.
\end{align*}
\end{proof}
We will use this to choose suitable values for the joint degree sequence of each $G_i$ so that each joint degree sequence sums to the same value and  %it adds up to the same value, 
making the $G_i$ have the same size.  %Therefore, 
To this end, \((d_1, \ldots, d_t)\) can be thought of as a partition of some integer $n$. The number of elements \(t\) is the same for all partitions; this is necessary to make spines of all caterpillars have the same length. For the \(i^\textrm{th}\) such partition we can define $G_i = \cat(d_1, \ldots, d_t;m)$. We will only consider the line graphs up to the \(m^{th}\) iteration, so the order of \(d_j\) does not matter. For any permutation \(\pi\), Graham sequences of \(\cat(\pi(d_1), \ldots, \pi(d_t);m)\) and  \(\cat(d_1, \ldots, d_t;m)\) are the same up to the \(m^\textrm{th}\) element, but the caterpillars can be different. In some sense, this is a deviation from the spirit of Graham's conjecture, since the claim is that every single tree produces a different sequence. %and all trees can be distinguished by their Graham sequences, while 
Our constraints produce large classes of caterpillars indistinguishable by the first $m$ terms of their Graham sequences. Nonetheless, this constraint is essential for our argument since analyzing iterated line graphs past the point where the legs at different joints interact is prohibitively difficult. 

Next we have to analyze the terms in the sum in Lemma \ref{cattostars}.

\begin{definition}
For any \(j\) and any \(S \subseteq V(\linegraph{j}(G))\), define the shadow \(\Sh(S)\) recursively as follows.
\[
\Sh(S) = \left \{ \begin{array}{l} S \textrm{ if } S \subseteq V(G)\\ \Sh (\bigcup\limits_{s \in S} s ) \textrm{ otherwise }\end{array} \right .
\]
\end{definition}

Note that $\Sh(v)\subseteq V(G)$ for any $v \in V(L^{(m)}(G))$ and any $m \geq 0$.  

\begin{lemma} \label{lem:shadowsize} For any graph $H$ and $m \geq 0$, $|\Sh(v)| \leq m+1$ for all $v \in V(L^{(m)}(H))$.
\end{lemma}
\begin{proof} Let \(\Sh_0 (v) = \{v\}, \Sh_{j + 1} (v) = \cup_{w \in \Sh_j (v) } w \). Then \(\Sh_m (v) = \Sh (v)\). Induction shows that \(\Sh_j (v)\) induces a connected subgraph in \(\linegraph{m-j} (H)\) for any \(j\).  To begin with, note that the statement is true for one vertex in \(\Sh_0 (v)\); assume that it is true for \(\Sh_{j-1} (v)\). For any \(u_1, u_2 \in \Sh_{j} (v)\) there are \(w_1, w_2 \in \Sh_{j-1} (v), u_1 \in w_1, u_2 \in w_2\). Since \(w_1\) and \(w_2\) are connected by a path in \(\Sh_{j-1} (v)\), \(u_1\) and \(u_2\) are connected by a path in \(\Sh_{j} (v)\) due to Proposition \ref{pathsinilg} and the fact that \(\Sh_{j-1} (v) \subseteq L(\Sh_{j} (v))\). 

Induction on \(j\) also yields that \(|\Sh_{j} (v)| \leq j + 1\) for all \(j\). This is immediate for \(\Sh_{0} (v)\). Assume it is true for \(\Sh_{j-1} (v)\). Since all edges of the subgraph induced by \(\Sh_{j} (v)\) are vertices of \(\Sh_{j-1} (v)\), \(\Sh_{j} (v)\) is a connected graph with no more than \(j\) edges, and therefore can not  have more than \(j + 1\) vertices.
\end{proof}

\begin{lemma}
If \(m > k\), then \(|L^{(k)}(S(d;m,m))| = f_k(d)\) is a polynomial of degree \(k\). 
\end{lemma}
\begin{proof}
We enumerate each isomorphism type of connected subgraphs of $S(d;m,m)$ containing the central vertex as $\{H_j\}_{j \in \mathcal{J}}$.  Denote the \textit{weight} of a graph $H$ by $\wt(H) = |\{v \in V(L^{(k)}(H)) : \Sh(v) = V(H)\}|$, i.e., the number of vertices in $L^{(k)}(H)$ that ``involve'' all vertices of $H$.  Then we have:
\begin{equation} \label{eq3}
|L^{(k)}(S(d;m,m))| = |L^{(k)}(P_{2m+1})| + \sum_{j \in \mathcal{J}} \wt(H_j) B_j.
\end{equation}
where
%\begin{equation}
\[
B_j = \left \{ \begin{array}{l} d \textrm{ if } H_j \cong S(1;0,0)  \\[.1in]
\binom{d+2}{2} \textrm{ if } H_j \cong S(2;0,0)  \\[.1in]
\binom{d+2}{a} \textrm{ if } H_j \cong S(a;0,0) \textrm{ for some $a \geq 1$} \\[.1in]
2 \binom{d+1}{a} \textrm{ if } H_j \cong S(a;b,0) \textrm{ for some } a \geq 2, \, b \geq 2 \\[.1in]
2 \binom{d}{a} \textrm{ if } H_j \cong S(a;b,c) \textrm{ for some } a \geq 1, \, b \geq 2, \, c \geq 2, \, b \neq c\\[.1in]
\binom{d}{a} \textrm{ if } H_j \cong S(a;b,b) \textrm{ for some } a \geq 1, \, b \geq 2. \end{array}\right .
\]
%\end{equation}
Note that the $H_j$ all have the form $S(a;b,c)$ for some \(a,b,c \geq 0\), such that \(a+b+c+1 \leq k+1\) (due to Lemma \ref{lem:shadowsize}), and $\wt(H_j)$ depends only on $H_j$, but not on \(d\). Each \(B_j\) is a polynomial in \(d\) of degree at most \(k\). Degree \(k\) is achieved only when \(a = k, b = 0, c = 0\), and the lead coefficient in this case is 1. %This proves the lemma.
\end{proof}

Lemma \ref{cattostars} combined with (\ref{eq3}) provides a count of the vertices of $L^{(m)}(G)$ and proves the equality (\ref{lgtopolyn}).

We use this fact to construct a large collection of caterpillars $\{G_i\}_{i \in \mathcal{I}}$ with the same number of vertices \(n\) such that, whenever $i \neq j$, there is such \(k < m\) that $|\linegraph{k}(G_i)| \neq |\linegraph{k}(G_j)|$. The cardinality of \(\mathcal{I}\) is a lower bound for the number of Graham classes of trees with \(n\) vertices.

We will need an upper bound on the size of the largest coefficient, and a lower bound on the size of the lead coefficient.  The rest of this section is dedicated to obtaining these bounds.

\begin{lemma} \label{lem:degreegrowth} If $G$ is a $d$-regular graph, then $L^{(k)}(G)$ is $(2^k d-2^{k+1}+2)$-regular.
\end{lemma}
\begin{proof} We proceed by induction.  The base case is almost immediate:  Given an edge $e \in E(G)$, its end-vertices each have degree $d$.  Therefore $e$ is incident to $d-1+d-1 = 2d-2$ edges $f$ in $G$, whence the degree of each vertex in $L(G)$ is $2d-2$.  Since $2d-2 = 2^1 d - 2^2 + 2$, we are done.  Now, suppose that $L^{(k)}(G)$ is $(2^k d-(2^{k+1}-2))$-regular.  By the base case, $L^{(k+1)}(G)$ is $(2 \cdot 2^k d-2 \cdot 2^{k+1}+4-2)$-regular.  However,
$$
2 \cdot 2^k d - 2 \cdot 2^{k+1} + 4 - 2 = 2^{k+1} d - 2^{k+2} + 2.
$$
\end{proof}

\begin{lemma}
For all \(k\) and \(n\), \(|\linegraph{k}(K_n)| \leq n^{k+1} 2^{k^2}\).
\end{lemma}
\begin{proof}
For any \(j\), by Lemma \ref{lem:degreegrowth}, \(|\linegraph{j}(K_n)| = \half (2^{j-1} (n-1) - 2^j + 2) |\linegraph{j-1}(K_n)|\). Therefore,
\begin{align*}
|L^{(k)}(K_n)|& = |K_n| \prod_{j=1}^{k} \frac{2^{j-1} ((n-1)-2) + 2}{2} = n \prod_{j=1}^{k} (2^{j-2} (n-3) + 1)\\
& \leq n \prod_{j=1}^{k} 2^{j-2} n < n^{k+1} 2^{k^2}.
\end{align*}
\end{proof}

\begin{cor} \label{cor:modstarsize} For \(k \geq 1\) and \(d \geq 3\), \(|L^{(k)}(S(d;a,b))| < (d+a+b)^k 2^{k^2}\).
\end{cor}
\begin{proof} 
Since \(|L(S(d;a,b))| = d+a+b\), we have \(L(S(d;a,b)) \subseteq K_{d+a+b}\).  Therefore, \(\linegraph{k}(S(d;a,b)) \subseteq \linegraph{k-1}(K_{d+a+b})\), and 
\[
|\linegraph{k}(S(d;a,b))| \leq |\linegraph{k-1}(K_{d+a+b})| \leq (d+a+b)^k 2^{(k-1)^2} \leq (d+a+b)^k 2^{k^2}.
\]
\end{proof}

We now need an upper bound on the number of terms present in expression (\ref{eq3}).  Recall that the $H_j$ range
over isomorphism classes of graphs which occur in the shadow of nodes in the $k^{\textrm{th}}$ iterated line graph.  We have shown that these graphs have at most $k+1$ vertices.

Given a polynomial $f$, we refer to the coefficient of $f$ which is the largest in absolute value as the ``maximum coefficient''. If $f$ is degree $k$, we refer to the coefficient of $x^k$ as the ``lead coefficient''. 

\begin{theorem} An upper bound on the maximum  coefficient of $f_k$ is $2^{6k^2}$ for \(k \geq 2\).
\end{theorem}
\begin{proof} Let the maximum coefficient of $f_k$ be $C$.  Going back to expression (\ref{eq3}), we see that
$$
C \leq |\mathcal{J}| \cdot \max_{j \in \mathcal{J}} \wt(H_j) \cdot \max_{j \in \mathcal{J}, \ell \in \mathbb{N}} [d^\ell] B_j.
$$
To bound the first factor, we count the isomorphism classes of graphs on $\leq k+1$ vertices (by Lemma \ref{lem:shadowsize}) which can be embedded into $S(d;m,m)$ and contain the central vertex.  Suppose $H_j = S(a;b,c)$; then $|H_j| = a + b + c + 1$.  Therefore, an upper bound for the number of elements of $\mathcal{J}$ is the number of nonnegative integer solutions to $a + b + c + 1 \leq k + 1$, i.e., the number of nonnegative integer solutions to $a+b+c+d = k$.  This is easily seen to be $\binom{k + 3}{3}$.

To bound the second factor, we employ Corollary \ref{cor:modstarsize}.  In particular, writing $H_j = S(d_j;a_j,b_j)$,
\begin{align*}
\max_{j \in \mathcal{J}} \wt(H_j) &= \max_{j \in \mathcal{J}} \wt(S(d_j;a_j,b_j)) \\
&< \max_{j \in \mathcal{J}} (d_j+a_j+b_j)^k 2^{k^2} \\
&\leq (k+1)^k 2^{k^2}
\end{align*}
by Lemma \ref{lem:shadowsize}.

To bound the third factor, we refer to the definition of $B_j$, which states that all \(B_j\) have the form \(K {n \choose t} = K \frac{1}{t!} n(n-1)\ldots (n-t+1) = \sum_{j=0}^t s(t, j) n^j\), where \(s(t,j)\) are signed Stirling numbers of the first kind and \(K\) is a constant which can be either 1 or 2. Since \(|s(t, j)|\) can be alternatively defined as the number of permutations of \([t]\) with \(j\) cycles, these numbers are always smaller than \(t!\), and, therefore, the coefficients of \(B_j\), considered as a polynomial in \(d\), are bounded by 2. Note that \(n\) can be either \(d\),
% in which case the bound follows immediately, or \((d+1)\) or \((d+2)\), in which case it's almost evident.
\(d+1\), or \(d+2\). In each case, the bound is clear. 
Putting the pieces together, we see that
\begin{align*}
C &\leq \binom{k+3}{3} \cdot (k+1)^k 2^{k^2} \cdot 2 \\
&\leq \frac{1}{6} (k+3)^3 (k+1)^k 2^{k^2 + 1} \\
&< (k+3)^{k+3} 2^{k^2 + 1}\\
&= 2^{(k+3)\log(k+3) + k^2 + 1}\\
& \leq 2^{(k+1)(k+3) + k^2 + 1}\\
& \leq 2^{k^2 +4k + 3 + k^2 + 1}\\
& = 2^{2k^2 + 4k + 4}\\
& \leq 2^{6k^2}.
\end{align*}

\end{proof}

\begin{cor} \label{cor:ratiobound} If \(d \geq k\), then an upper bound on the ratio of the maximum coefficient to the lead coefficient of $f_k(d)$ is
\[
2^{6k^2} k!
\]
\end{cor}
\begin{proof}
All \(B_j\) have either the form \(2 {d \choose a}\) or \({d \choose a}\). The lead coefficient of \(B_j\) is nonzero only if \(a = k\), which is possible if \(d \geq k\) and \(H_j \cong S(a;0,0)\). In this case the coefficient is \(1/k!\) or \(2/k!\).  By (\ref{eq3}) and the fact that \(\wt(H_j)\) is a positive integer, all contributions to the lead coefficient of \(f_k(d)\) are nonnegative and at least \(1/k!\), so the lead coefficient of \(f_k(d)\) is as well.
\end{proof}

\section{Sums of Powers of Parts}
For a finite sequence of integers $A = \{a_i \}_{i=1}^n$, let $S_r(A) = \sum_{i = 1}^n a_i^r$, and for $t \in \Z$ let \(A + t = \{a_i + t\}_{i=1}^n\). For any function \(f\) let \(f(A) = \sum_{i=0}^n f(a_i)\). The product of two sequences will be interpreted as concatenation, i.e. if \(A = \{a_i \}_{i=1}^n\) and \(B = \{b_i \}_{i=1}^m\), then \(AB = (a_1, \ldots, a_n, b_1, \ldots, b_m)\).

Define two parametric families of sequences $\mathbf{T}_j$ and $\overline{\mathbf{T}}_j$ as follows.
\begin{enumerate}
\item $\mathbf{T}_0 = \emptyset$
\item $\overline{\mathbf{T}}_j = (0,\ldots,2^{j}-1) \setminus \mathbf{T}_j$
\item $\mathbf{T}_{j+1} = \mathbf{T}_j \left ( \overline{\mathbf{T}}_j + 2^j \right )$
\end{enumerate}

In other words, $\mathbf{T}_j$ and $\overline{\mathbf{T}}_j$ are subsequences of $(0,\ldots,2^j-1)$, and the parity of the sum of any of these numbers' binary digits determines to which sequence it belongs. If the sum is odd, the number belongs to $\mathbf{T}_j$, and if it is even, the number belongs to $\overline{\mathbf{T}}_j$. Both sequences are increasing.

It has been known since 1851 (\cite{P51}) that
\begin{equation}\label{eq1}
S_r(\mathbf{T}_k)=S_r(\overline{\mathbf{T}}_k)
\end{equation}
when $k > r$, i.e., the pair $(\mathbf{T}_k,\overline{\mathbf{T}}_k)$ provides a solution to the degree-$r$ Prouhet-Tarry-Escott problem (q.v.~\cite{B02}). We will need an extended version of this equality. 

\begin{lemma}\label{prouhetgen}
For any \(k, r\) such that \(k > r\) and any \(t \in \R\)
\[
S_r(\mathbf{T}_k + t) - S_r(\overline{\mathbf{T}}_{k} + t) = 0.
\]
\end{lemma}

\begin{proof}
\begin{align*}
S_r(\mathbf{T}_k + t) - S_r(\overline{\mathbf{T}}_k + t) 
& = \sum_{x \in \mathbf{T}_k} (x + t)^r - \sum_{x \in \overline{\mathbf{T}}_k} (x + t)^r \\
& = \sum_{x \in \mathbf{T}_k} \sum_{i=0}^r \binom{r}{i} x^i t^{r - i} - \sum_{x \in \overline{\mathbf{T}}_k} \sum_{i=0}^r \binom{r}{i} x^i t^{r - i} \\
& = \sum_{i=0}^r \binom{r}{i} t^{r - i} \left( \sum_{x \in \mathbf{T}_k} x^i - \sum_{x \in \overline{\mathbf{T}}_k} x^i \right)\\
& = \sum_{i=0}^r \binom{r}{i} t^{r - i} \left( S_i(\mathbf{T}_k) - S_i(\overline{\mathbf{T}}_k) \right) = 0
\end{align*}
\end{proof}

When $k \geq r$, the conclusion of Lemma \ref{prouhetgen} is no longer true.

\begin{prop} $S_k(\mathbf{T}_k)-S_k(\overline{\mathbf{T}}_k) = (-1)^{k+1} k! 2^{\binom{k}{2}}$ for $k \geq 1$.  Furthermore,\label{prop:funsum}
$$
\left | S_{r}(\mathbf{T}_k) - S_{r}(\overline{\mathbf{T}}_k) \right | \leq 2^{k(r+1)}
$$
for all $r \geq 0$.
\end{prop}

\begin{proof} We begin with the first statement, and proceed by induction.  For $k=1$,
$$
S_1(\mathbf{T}_1)-S_1(\overline{\mathbf{T}}_1) = 1^1 - 0^1 = 1 = (-1)^{1+1} 1! 2^{\binom{1}{2}}.
$$
Suppose the statement is true for $k-1$.  Then we may write
\begin{align*}
S_k(\mathbf{T}_k)-S_k(\overline{\mathbf{T}}_k) &= S_k(\mathbf{T}_{k-1})-S_k(\overline{\mathbf{T}}_{k-1}) +\\
& S_k(2^{k-1} +\overline{\mathbf{T}}_{k-1}) - S_k(2^{k-1} + \mathbf{T}_{k-1}) = \\
& S_k(\mathbf{T}_{k-1})-S_k(\overline{\mathbf{T}}_{k-1}) + \sum_{j=0}^k \binom{k}{j}2^{(k-1)j}S_{k-j}(\overline{\mathbf{T}}_{k-1})\\
&\qquad - \sum_{j=0}^k \binom{k-1}{j}2^{kj} S_{k-j}(\mathbf{T}_{k-1})
\end{align*}
by the Binomial Theorem.  Therefore,
\begin{align*}
S_k(\mathbf{T}_k)-S_k(\overline{\mathbf{T}}_k) &= \sum_{j=1}^k \binom{k}{j}2^{(k-1)j}S_{k-j}(\overline{\mathbf{T}}_{k-1}) - \sum_{j=1}^k \binom{k}{j}2^{(k-1)j} S_{k-j}(\mathbf{T}_{k-1})\\
&= k 2^{k-1} \left ( S_{k-1}(\overline{\mathbf{T}}_{k-1}) - S_{k-1}(\mathbf{T}_{k-1})\right ),
\end{align*}
since, by (\ref{eq1}), all terms with $j > 1$ are zero.  Applying the inductive hypothesis, we obtain
\begin{align*}
S_k(\mathbf{T}_k)-S_k(\overline{\mathbf{T}}_k) &= - k 2^{k-1} (-1)^{k}(k-1)! 2^{\binom{k-1}{2}} \\
&= (-1)^{k+1} k! 2^{\binom{k}{2}}.
\end{align*}

To see the second part of the statement, simply note that there are fewer than $2^k$ elements  of $\mathbf{T}_k$ (resp. $\overline{\mathbf{T}}_k$), each of which is at most $2^k$. Hence summing $r$th powers of the elements of $\mathbf{T}_k$ (resp. $\overline{\mathbf{T}}_k$) is at most $2^k(2^{kr}) = 2^{k(r+1)}$, providing the desired bound.  

\end{proof}

Before we can prove Theorem \ref{final}, we need some results (Corollary \ref{lemma:coeff}, Proposition \ref{prop:polydiff} ,Corollary \ref{cor:polycancel}, and Lemma \ref{lem:monotone}) about arbitrary polynomials. 

\begin{cor} \label{lemma:coeff} For a polynomial \(f\) of degree \(r\),  let
\[
g_k(t) = f(\mathbf{T}_k + t) - f(\overline{\mathbf{T}}_{k}+t).
\]
Then, if \(k > r\), we have \(g_k(t) = 0\) for any \(t \in \Z\). If \(k \leq r\), then \(g_k(t)\) is a polynomial of degree \(k - r\). If $C$ is the lead coefficient of $f$, and $C^\prime$ is the largest non-lead coefficient of $f$, then the ratio of the lead coefficient of $g$ and the sum of the rest of the coefficients is at most \(\frac{2^{5r^2}}{k!} \frac{C}{C^\prime}\)
\end{cor}
\begin{proof} Suppose $f(x) = \sum_{j=0}^r a_r x^r$. Then
\[
f(\mathbf{T}_k + t) = \sum_{x \in \mathbf{T}_k} f(x+t) = \sum_{x \in \mathbf{T}_k} \sum_{j=0}^r a_j (x+t)^j =\sum_{j=0}^r a_j S_j(\mathbf{T}_k + t).
\]
Similarly,
\begin{align*}
f(\overline{\mathbf{T}}_k + t) &= \sum_{j=0}^r \sum_{i=0}^j a_j  S_j(\overline{\mathbf{T}}_k + t).
\end{align*}
Consider the case \(k > r\). It follows from Lemma \ref{prouhetgen}, that
\[
f(\mathbf{T}_k + t) - f(\overline{\mathbf{T}}_k + t) = \sum_{j=0}^r a_j (S_j(\mathbf{T}_k + t) -  S_j(\overline{\mathbf{T}}_k + t)) = 0
\]
In the case \(k \leq r\),
\begin{align*}
f(\mathbf{T}_k + t) - f(\overline{\mathbf{T}}_k + t) &= \sum_{j=0}^r \sum_{i=0}^j a_j \binom{j}{i} t^{j-i} \left ( \sum_{x \in \mathbf{T}_k} x^i - \sum_{x \in \overline{\mathbf{T}}_k} x^i \right ) \\
&= \sum_{j=k}^r \sum_{i=k}^j a_j \binom{j}{i}  t^{j-i} \left ( \sum_{x \in \mathbf{T}_k} x^i - \sum_{x \in \overline{\mathbf{T}}_k} x^i \right ) \\
&= \sum_{q=0}^{r-k} t^q \sum_{j=q+k}^r a_j \binom{j}{q} \left ( \sum_{x \in \mathbf{T}_k} x^{j-q} - \sum_{x \in \overline{\mathbf{T}}_k} x^{j-q} \right ),
\end{align*}
where the second equality follows from the fact that the pair $\{\mathbf{T}_k,\overline{\mathbf{T}}_k\}$ is a solution to the Prouhet-Tarry-Escott problem of any order $i < k$.  To complete the proof, we need to show that the coefficient $c_{r-k}$ of $t^{r-k}$ is nonzero.  However,
\[
c_{r-k} = \binom{r}{k} \left ( \sum_{x \in \mathbf{T}_k} x^{k} - \sum_{x \in \overline{\mathbf{T}}_k} x^{k} \right )a_r = (-1)^{k+1} \binom{r}{k} k! 2^{\binom{k}{2}} a_r\neq 0,
\]
by Proposition \ref{prop:funsum}. For the proof of the second part of the Lemma, we note that the sum of the non-lead coefficients of $g$ is at most the largest non-lead coefficient of $f$ multiplied by
\begin{align*}
\left | \sum_{q=0}^{r-k-1} \sum_{j=q+k}^{r} \binom{j}{q} \left ( S_{j-q}(\mathbf{T}_k) - S_{j-q}(\overline{\mathbf{T}}_k) \right ) \right | & \leq \sum_{q=0}^{r-k-1} \sum_{j=q+k}^{r} \binom{j}{q} 2^{k(j-q+1)} \\
& \leq \sum_{q=0}^{r-k-1} \sum_{j=q+k}^r\frac{j^q}{q!} 2^{k(j-q+1)}\\
& < r\sum_{q=0}^{r-k-1} \frac{r^q}{q!} 2^{k(r-q+1)} \\
& < r 2^{k(r+1)} \sum_{q = 0}^\infty \frac{r^q}{q!} \\
& = r 2^{k(r+1)} e^r,
\end{align*}
where the first inequality appeals to the second part of Proposition \ref{prop:funsum}.  Therefore, the desired ratio is at most
\[
\frac{C}{C^\prime} \frac{r 2^{k(r+1)} e^r}{\binom{r}{k} k! 2^{\binom{k}{2}}} \leq \frac{C}{C^\prime} \frac{r 2^{k(r+1)} e^r}{k!} \leq \frac{C}{C^\prime} \frac{2^{\log (r)} 2^{k(r+1)} 4^r}{k!} \leq \frac{2^{5r^2}}{k!} \frac{C}{C^\prime}.
\]
\end{proof}

\begin{theorem}
If $h(x)$ is a polynomial with lead coefficient at least $N$ in absolute value, and the sum of absolute values of the rest of the coefficients   is at most $M$, then $h(x)$ is strictly monotone on the interval $(A, \infty)$, where \(A = \max(1, M/N)\). 
\end{theorem}

\begin{proof}
Let \(h(x) = \sum_{j=0}^d a_j x^j\). Assume that the lead coefficient of $h$ is positive, \(a_d = N\). We show that the first derivative of $h(x)$ is strictly positive on the interval $(A,\infty)$. 
\begin{align*}
h^\prime(x) & = \sum\limits_{j=1}^d j a_j x^{j-1} \geq  d a_d x^{d-1} - \sum\limits_{j=1}^{d-1} j |a_j| x^{j-1} \\
&\geq  d a_d x^{d-1} - (d-1) x^{d-2} \sum\limits_{j=1}^{d-1} |a_j| \\
& \geq N d x^{d-1} - (d-1) M x^{d-2}\\
& > 0,
\end{align*}
provided that \(x > M/N > ((d-1) M)/(N d)\) and \(x > 1\). If the lead coefficient of $h$ is negative, multiply $h$ by $(-1)$ and apply the above argument. \(h(x)\) in this case is decreasing for \(x > A\). 
\end{proof}

\begin{cor}\label{cor:increasing}
If \(m < K, k < K\), \(f_m(d) = |L^{(m)}(S(d;K,K))|\), and \(g(t) = f(\mathbf{T}_k + t) - f(\overline{\mathbf{T}}_{k}+t)\), then \(g(t)\) is monotone for \(t > 2^{11 K^2}\). 
\end{cor}
\begin{proof}
Due to Corollary \ref{cor:ratiobound}, \(\frac{C}{C^\prime} < k! 2^{6k^2}\). Therefore, if \(N\) is the lead coefficient of \(g\), and \(M\) is the sum of the rest of the coefficients, then \(\frac{M}{N} \leq \frac{k! 2^{6k^2} 2^{5K^2}}{k!} \leq 2^{11 K^2}\).
\end{proof}

For $r \geq 0$, $k \geq 0$ and $s \geq t \geq 0$, define the sequence $\mathbf{W}(k;r,s,t)$ as follows.

\begin{align*}
\mathbf{W}(k;r,s,t) &= (\overline{\mathbf{T}}_k)^r (\mathbf{T}_k)^{s}\left (\prod_{j=1}^{t} (\overline{\mathbf{T}}_k+j2^k) \right )(\mathbf{T}_k+(t+1)2^k) \\
& \qquad \prod_{j=1}^{s-t} (\overline{\mathbf{T}}_k+(j+t+1)2^k)\prod_{j=1}^{r}(\mathbf{T}_k+(j+s+1)2^k)
\end{align*}
where the empty product is interpreted as the empty sequence. For example \(\mathbf{T}_2 = (1, 2)\) and \(\overline{\mathbf{T}}_2 = (0, 3)\), so
\[
\mathbf{W}(2;2,2,1) = (0,3,0,3,1,2,1,2,4,7,9,10,12,15,17,18,21,22)
\]

\begin{prop}
$\mathbf{W}(k;r,s,t)$ is a partition of \(4^{k-1}((r + s)^2 +5r + 5s + 3) -2^{k-2}(2r+2s+1)\) consisting of $2^{k-1}(2r + 2s+1)$ parts for \(k \geq 2\).
\end{prop}
\begin{proof}
It follows from (\ref{eq1}) with \(r = 1\) that the sums of the elements of $\mathbf{T}_k$ and $\overline{\mathbf{T}}_k$ are the same and, therefore, are equal to $4^{k-1}-2^{k-2} = B$ for $k \geq 2$. Also, it can be proved by induction that, for \(k \geq 1\), \(\mathbf{T}_k\) and \(\overline{\mathbf{T}}_k\) have the same number of elements, which is \(2^{k-1}\). Note that if a number \(a\) is added to the sequence \(\mathbf{T}_k\) (or \(\overline{\mathbf{T}}_k\)), the sum of elements will increase by \(a 2^{k-1}\). Therefore, for $k \geq 2$, $\mathbf{W}(k;r,s,t)$ is a partition of 
\begin{align*}
& rB + sB + \sum_{j=1}^t (B+2^{k-1}j2^k) + (B + 2^{k-1}(t+1)2^k) \\
& + \sum_{j=1}^{s-t} (B + 2^{k-1}(t+j+1)2^k) + \sum_{j=1}^r (B + 2^{k-1}(j+s+1)2^k) \\
& = B(r + s + t + 1 + s - t + r) \\
& \qquad + 2^{2k-1} \left (\sum_{j=1}^t j + (t+1) + \sum_{j=1}^{s-t} (t+j+1) + \sum_{j=1}^r (j+s+1) \right ) \\
& = B(2r + 2s + 1) \\
& \qquad + 2^{2k-1} \left ( t+1 + (t+1)(s-t) + (s+1)r + \sum_{j=1}^t j + \sum_{j=1}^{s-t} j + \sum_{j=1}^r j \right ) \\
& = B(2r + 2s + 1) + 2^{2k-1} \left ( (t+1)(s-t+1) + (s+1)r \right . \\
& \left . \qquad + \frac{t(t+1)}{2}+ \frac{(s-t)(s-t+1)}{2}+ \frac{r(r+1)}{2} \right ) \\
& = (4^{k-1}-2^{k-2})(2r + 2s + 1) + 2^{2k-1} (s^2/2 +r^2/2 +3r/2+3s/2+rs +1)\\
& = 4^{k-1}(2r + 2s +1) + 4^{k-1}(s^2 + r^2 +3r+3s+2rs +2) - 2^{k-2} (2r+2s+1)\\
&= 4^{k-1}((r + s)^2 +5r + 5s + 3) -2^{k-2}(2r+2s+1).
\end{align*}
The number of parts in the partition represented by $\mathbf{W}(k;r,s,t)$ can be calculated directly from the definition and the number of parts in \(\mathbf{T}_k\) and \(\overline{\mathbf{T}}_k\).
\end{proof}

Define $\mathcal{W}^{ks}_j$ for $1 \leq j < (s+2)(s+1)/2$ to be the $j$-th element of the sequence:
\begin{align*}
& \mathbf{W}(k;0,s,0), \mathbf{W}(k;0,s,1), \mathbf{W}(k;0,s,2), \ldots, \mathbf{W}(k;0,s,s),\\
&  \mathbf{W}(k;1,s-1,0), \mathbf{W}(k;1,s-1,1), \ldots, \mathbf{W}(k;1,s-1,s-1),\\
&  \mathbf{W}(k;2,s-2,0), \mathbf{W}(k;2,s-2,1),  \ldots, \mathbf{W}(k;2,s-2,s-2),\\
& \qquad \vdots \\
&  \mathbf{W}(k;s-1,1,0),  \mathbf{W}(k;s-1,1,1), \\
&  \mathbf{W}(k;s,0,0).
\end{align*}
Note that each of these $\mathcal{W}^{ks}_j$ is a partition of $4^{k-1}(s^2+5s+3) - 2^{k-2}(2s+1)$ of length $2^{k-1}(2s+1)$. 

\begin{prop} \label{prop:polydiff} 
For any polynomial $f$, let  \(g(t) = f(\mathbf{T}_k + t) - f(\overline{\mathbf{T}}_{k}+t)\). Then
\[
f(\mathbf{W}(k;r,s,t+1))-f(\mathbf{W}(k;r,s,t)) = g((t+2)2^k) - g((t+1)2^k)
\]
In addition, for $r \leq s - 1$,
\[
f(\mathbf{W}(k;r+1,s-1,0))-f(\mathbf{W}(k;r,s,s)) = g(2^k)-g(0)
\]
\end{prop}
\begin{proof}
For any polynomial $f$, note that there will be some cancellation in the difference $f(\mathbf{W}(k;r,s,t+1)) -f(\mathbf{W}(k;r,s,t))$ because $\mathbf{W}(k; r,s,t+1)$ and $\mathbf{W}(k;r,s,t)$ share a common prefix of $(\overline{\mathbf{T}}_k)^r (\mathbf{T}_k)^s$, 
and a common suffix of $\prod_{j=1}^r (\mathbf{T}_k + (j + s +1) 2^k)$. We focus now on the remaining terms. The middle terms of $\mathbf{W}(k;r,s,t)$ are of the form:
\begin{align*}
&\left (\prod_{j=1}^{t} (\overline{\mathbf{T}}_k+j2^k) \right )(\mathbf{T}_k+(t+1)2^k)   \prod_{j=1}^{s-t} (\overline{\mathbf{T}}_k+(j+t+1)2^k)\\
&= (\overline{\mathbf{T}}_k + 2^k) (\overline{\mathbf{T}}_k +2 \cdot 2^k) \ldots (\overline{\mathbf{T}}_k + t2^k) \\
& (\mathbf{T}_k + (t+1) 2^k) (\overline{\mathbf{T}}_k+ (t+2)2^k) \ldots (\overline{\mathbf{T}}_k+ (s+1)2^k).
\end{align*}
The middle terms of $\mathbf{W}(k;r,s,t+1)$ are of the form:
\begin{align*}
&\left (\prod_{j=1}^{t+1} (\overline{\mathbf{T}}_k+j2^k) \right )(\mathbf{T}_k+(t+2)2^k)   \prod_{j=1}^{s-(t+1)} (\overline{\mathbf{T}}_k+(j+(t+1)+1)2^k)\\
&= (\overline{\mathbf{T}}_k + 2^k) (\overline{\mathbf{T}}_k +2 \cdot 2^k) \ldots (\overline{\mathbf{T}}_k + (t+1)2^k) \\
& (\mathbf{T}_k + (t+2) 2^k) (\overline{\mathbf{T}}_k+ (t+3)2^k) \ldots (\overline{\mathbf{T}}_k+ (s+1)2^k).
\end{align*}
Again, many terms cancel. In particular, we have:
\begin{align*}
f(\mathbf{W}(k;r,s,t+1)) - f(\mathbf{W}(k;r,s,t)) &= f(\overline{\mathbf{T}}_k + (t+1)2^k) + f(\mathbf{T}_k + (t+2)2^k)\\&- f(\mathbf{T}_k   + (t+1)2^k)-f(\overline{\mathbf{T}}_k + (t+2)2^k)\\
& = g_k((t+2)2^k) -g_k((t+1)2^k).
\end{align*}
This verifies the first part of the proposition.     

The second part follows from an application of Corollary \ref{lemma:coeff}, and the observation that:
\begin{align*}
\mathbf{W}(k;r+1,s-r-1, 0) &= (\overline{\mathbf{T}}_k)^{r+1} (\mathbf{T}_k)^{s-r-1}(\mathbf{T}_k + 2^k)\\
&  \prod_{j=1}^{s-r-1}(\overline{\mathbf{T}}_k +(j+1)2^k) \prod_{j=1}^{r+1}(\mathbf{T}_k + (j+s-r)2^k)
\end{align*}
and 
\begin{align*}
\mathbf{W}(k;r,s-r,s-r)& =( \overline{\mathbf{T}}_k)^r (\mathbf{T}_k)^{s-r} \left( \prod_{j=1}^{s-r}(\overline{\mathbf{T}}_k +j2^k) \right) \\
& (\mathbf{T}_k + (s-r+1)2^k)\prod_{j=1}^r(\mathbf{T}_k+(j+s-r+1)2^k).
\end{align*}
Hence we have
\begin{align*}
& f(\mathbf{W}(k;r,s-r-1,0)) -f(\mathbf{W}(k;r,s-r,s-r))\\
& = f(\overline{\mathbf{T}}_k) - f(\mathbf{T}_k) + f(\mathbf{T}_k + 2^k) - f(\overline{\mathbf{T}}_k +2^k)\\
& = g_k(2^k) -g_k(0)
\end{align*}
which completes the proof, after substituting $s-r$ for $s$. 

\end{proof}

Since the polynomials \(f_k\) that constitute the formula for the size of the \(k^\textrm{th}\) iterated line graph are increasing only after a certain point, it will be useful to analyze $f$ with its variable shifted by an additive constant. 

\begin{cor}
\label{cor:polycancel}
For any polynomial $f$, let \(g(t) = f(\mathbf{T}_k + t) - f(\overline{\mathbf{T}}_{k}+t)\). Then
\[
f(\mathbf{W}(k;r,s,t+1)+A)-f(\mathbf{W}(k;r,s,t)+A) = g((t+2)2^k+A) - g((t+1)2^k+A)
\]
In addition, for $r \leq s - 1$,
\[
f(\mathbf{W}(k;r+1,s-1,0)+A)-f(\mathbf{W}(k;r,s,s)+A) = g(2^k+A)-g(A)
\]
\end{cor}

\begin{proof}
The proof is identical to that of Proposition (\ref{prop:polydiff}).
\end{proof}

\begin{theorem} \label{thm:monotoneventually}
If $h(x)$ is a polynomial with lead coefficient at least $N$ in absolute value, and the sum of absolute values of the rest of the coefficients   is at most $M$, then $h(x)$ is strictly monotone on the interval $(A, \infty)$, where \(A = \max(1, M/N)\). 
\end{theorem}

\begin{proof}
Let \(h(x) = \sum_{j=0}^d a_j x^j\). Assume that the lead coefficient of $h$ is positive, \(a_d = N\). We show that the first derivative of $h(x)$ is strictly positive on the interval $(A,\infty)$. 
\begin{align*}
h^\prime(x) & = \sum\limits_{j=1}^d j a_j x^{j-1} \geq  d a_d x^{d-1} - \sum\limits_{j=1}^{d-1} j |a_j| x^{j-1} \\
&\geq  d a_d x^{d-1} - (d-1) x^{d-2} \sum\limits_{j=1}^{d-1} |a_j| \\
& \geq N d x^{d-1} - (d-1) M x^{d-2}\\
& > 0,
\end{align*}
provided that \(x > M/N > ((d-1) M)/(N d)\) and \(x > 1\). If the lead coefficient of $h$ is negative, multiply $h$ by $-1$ and apply the above argument. \(h(x)\) in this case is decreasing for \(x > A\). 
\end{proof}

\begin{lemma} \label{lem:monotone}
Consider any polynomial $f$ of degree \(r\), and let \(A\) be as in the statement of Theorem \ref{thm:monotoneventually}. If \(r \geq k\), then the sequence $f(\mathcal{W}^{ks}_j + A)$ is strictly monotone in \(j\) for $1 \leq j \leq \frac{(s+2)(s+1)}{2}$ and any \(s \geq 1, k \geq 2\). If \(r < k\), then this sequence is constant for any \(A\). 
\end{lemma}

\begin{proof}
Let  \(g(t) = f(\mathbf{T}_k + t) - f(\overline{\mathbf{T}}_{k}+t)\). Consider the case \(r \geq k\). Since \(g\) is a polynomial, there is an \(A\) so that \(g(x)\) is strictly monotone for \(x > A\). Consider the case when \(g(x)\) is increasing. The decreasing case is similar.  For any $j$, $\mathcal{W}^{ks}_j$ is of the form $\mathbf{W}(k;u, s-u, t)$ for some $s \geq u \geq 0$ and $s-u \geq t \geq 0$. We have two cases.

\textit{Case 1}. $\mathcal{W}_{j+1}^{ks}$ is of the form $\mathbf{W}(k;u,s-u,t+1)$ (corresponding to a change within a row in the array). In this case, Corollary \ref{cor:polycancel} tells us that \(f(\mathcal{W}_{j+1}^{ks} +A)-f(\mathcal{W}^{ks}_j +A) = g((t+2)2^k +A) - g((t+1)2^k +A) > 0 \), since \((t+2)2^k +A > (t+1)2^k +A\).

\textit{Case 2}. $\mathcal{W}_{j+1}^{ks}$ is of the form $\mathbf{W}(k;u+1,s-u-1,0)$ (corresponding to a transition down one row in the array). Corollary \ref{cor:polycancel} tells us that \(f(\mathcal{W}_{j+1}^{ks} + A2^k) - f(\mathcal{W}^{ks}_j + A2^k ) = g(2^k + A) - g(0 + A) > 0\).

The case when \(r < k\) follows from the same considerations and the fact that \(g(t) = 0\) for all \(t\), which follows from Corollary \ref{lemma:coeff}. 

\end{proof}
We now specialize the results of Corollary \ref{lemma:coeff}, Proposition \ref{prop:polydiff}, Corollary \ref{cor:polycancel}, and Lemma \ref{lem:monotone} to polynomials which are the size of some iterated line graphs. 
\begin{cor}\label{cor:increasing2}
If \(m < K, k < K\), \(f_m(d) = |L^{(m)}(S(d;K,K))|\), and \(g(t) = f_m(\mathbf{T}_k + t) - f_m(\overline{\mathbf{T}}_{k}+t)\), then \(g(t)\) is monotone for \(t > 2^{11 K^2}\). 
\end{cor}
\begin{proof}
Let \(C\) be the lead coefficient of \(f_m\), and \(C^\prime\) its largest non-lead coefficient.  Due to Corollary \ref{cor:ratiobound}, \(\frac{C}{C^\prime} < k! 2^{6k^2}\). Therefore, if \(N\) is the lead coefficient of \(g\), and \(M\) is the sum of the absolute values of the rest of the coefficients, then \(\frac{M}{N} \leq \frac{k! 2^{6k^2} 2^{5K^2}}{k!} \leq 2^{11 K^2}\).
\end{proof}

We are now ready to prove  Theorem \ref{final}.
 
\begin{theorem}\label{final}
For each $K >0$ and for some constant \(C\), there exists an \(N_0 \leq 2^{C K^2}\) such that for any \(N \geq N_0\), there are $\Omega(N^{K-1})$ distinct Graham classes of trees on $N$ vertices. 
\end{theorem} 
\begin{proof}
%\begin{comment}
%For each $m$,  there exists a degree \(m\) polynomial $f_m$ such that if $\{d_i\}_{i=1}^\ell $ are positive integers, 
%then $|L^{(m)} (\cat(d_1, \ldots, d_\ell;k))| = \sum_{i=1}^\ell f_m(d_i)$ for \(k \geq m\) by Lemma \ref{cattostars}. Fix $K > 0$ and each such polynomial and consider  $\{f_2, f_3, \ldots, f_K\}$. We will produce an $N = N(K)$ and a collection of partitions  of $N$, $\{\lambda_j^{ks}\}$ such that $ |\{f_h(\lambda_j^{ks})\}_{h \leq K}| = \Omega(N^{K-1})$. Each partition will define a caterpillar in the following way. Let \(\lambda_j^{ks} = (d_1, \ldots, d_p)\) then the corresponding caterpillar will be \(\cat(d_1, \ldots, d_p; k)\). Note that not only \(\lambda_j^{ks}\) are partitions of the same number for all \(j\), but they also have the same number of elements, so all caterpillars have the same length, which makes them different trees of the same size.
 
%Since all these caterpillars are of the same size, they also have the same number of edges, which is the size of \(\linegraph{1}(G)\). Therefore, it's only reasonable to look at the line graph sizes starting with \(\linegraph{2}(G)\). 
%\end{comment}
 
For each $2 \leq k \leq K$, there exists such number $A_k$ that the sequence $\{f_k(W_j^{ks} + A_k)\}$ is strictly monotone (without loss of generality assume that it is  increasing) in $j$. If $A = \max_{2 \leq k\leq K} A_k$, then, due to Lemma \ref{lem:monotone} and Corollary \ref{cor:increasing2}, \(A < 2^{11K^2}\). We write $\lambda_j^{ks} = \mathcal{W}^{ks}_j + A$.  Observe that $\lambda_j^{ks} \vdash$ $4^{k-1}(s^2+5s+3) - 2^{k-2}(2s+1)(1 - 2 A)= n_{ks}$ for $2 \leq k \leq K$ and for $1\leq j\leq \frac{(s+2)(s+1)}{2}$.  In particular, when $k$ is held constant, $n_{ks} = O(s^2)$. 
 
Let $\Lambda^{Ks}$ denote the collection of partitions of the form \(\prod_{i=2}^K \lambda^{is}_{j_i}\) for all possible choices of indices \(j_i\). In other words, $\Lambda^{Ks}$ is the collection of partitions which are concatenations of precisely one $\lambda_{j}^{is}$ for each $2 \leq i \leq K$. Observe that every element of $\Lambda^{Ks}$ is a partition of $\sum_{k=2}^{K} n_{ks} = N$. In particular, each element of $\Lambda^K$ forms an ordered partition of the {\em same} $N$, which (when $K$ is constant) is $O(s^2)$. When \(s\) is constant, \(N = O(2^{C K^2})\) for some constant \(C\). 

There are exactly \(N^{K-1}\) elements in the family \(\Lambda^{Ks}\), so if we can prove that for any distinct \(\lambda_1, \lambda_2 \in \Lambda^{Ks}\) the caterpillars \(\cat(\lambda_1; K)\) and \(\cat(\lambda_2; K)\) produce different Graham sequences, then  Theorem \ref{final} is proved. Due to Corollary  \ref{cor:polycancel}, this is equivalent to \(g_k(\lambda_1) \not = g_k(\lambda_2)\) for some \(k \leq K\). 

For \(\lambda \in \Lambda^{Ks}\) denote the sequence \((f_2(\lambda), \ldots, f_K(\lambda))\) as \(F(\lambda)\). We need to prove that for any \(\lambda_1, \lambda_2 \in \Lambda^{Ks}, \lambda_1 \not = \lambda_2\) the sequences \(F(\lambda_1)\) and \(F(\lambda_2)\) are different. Let \(\lambda_i = \lambda_{j_{Ki}}^{Ks} \lambda_{j_{(K-1)i}}^{(K-1)s} \ldots \lambda_{j_{3i}}^{3s} \lambda_{j_{2i}}^{2s}\) for \(i = 1, 2\). It can be proved by induction that if there is such \(k\) that \(\lambda_{j_{k1}}^{ks} \not = \lambda_{j_{k2}}^{ks}\) then \(f_k(\lambda_1) \not = f_k(\lambda_2)\). The base case for the induction is \(k = K\). It follows from Lemma \ref{lem:monotone} that \(f_K(\lambda_{j_{K1}}^{Ks}) \not = f_K(\lambda_{j_{K2}}^{Ks})\) and for any \(i \leq K\), \(f_K(\lambda_{j_{i1}}^{is}) = f_K(\lambda_{j_{i2}}^{is})\). Therefore, \(f_K(\lambda_1) \not = f_K(\lambda_2)\). 

Assume Theorem \ref{final} is  proved for \(k = K, \ldots, p + 1\). This covers all cases when \(\lambda_{j_{k1}}^{ks} \not = \lambda_{j_{k2}}^{ks}\) for \(k = K, \ldots, p + 1\), therefore we can assume that \(\lambda_{j_{k1}}^{ks} = \lambda_{j_{k2}}^{ks}\) for these \(k\), and \(\lambda_{j_{p1}}^{ps} \not = \lambda_{j_{p2}}^{ps}\). Then 

\begin{align*}
f_p(\lambda_1) - f_p(\lambda_2) 
&= \sum\limits_{k=p+1}^K (f_p(\lambda_{j_{k1}}^{ks}) - f_p(\lambda_{j_{k2}}^{ks})) + f_p(\lambda_{j_{p1}}^{ps}) - f_p(\lambda_{j_{p2}}^{ps})\\
&+ \sum\limits_{k=2}^{p-1} (f_p(\lambda_{j_{k1}}^{ks}) - f_p(\lambda_{j_{k2}}^{ks})) = f_p(\lambda_{j_{p1}}^{ps}) - f_p(\lambda_{j_{p2}}^{ps}) \not = 0, 
\end{align*}
where the first summand is equal to 0 due to the inductive assumption, and the second due to Lemma \ref{lem:monotone}.
\end{proof}

Note that we take \(k \geq 2\) because the size of the first line graph of a tree is completely determined by the size of the tree. 

\begin{cor}
For any \(N\) there are \(e^{\Omega((\log N)^{3/2})}\) Graham classes of trees on \(N\) vertices.
\end{cor}
\begin{proof}
Due to Theorem \ref{final}, for any \(K\) there exists \(N = O(2^{C K^2})\) so that there are \(\Omega(N^{K-1})\) distinct Graham classes of trees on \(N\) vertices. Therefore, there are \(\Omega(N^{\Omega(\log(N)^{1/2})}) = e^{\Omega((\log N)^{3/2})}\) distinct Graham classes of trees on \(N\) vertices.
\end{proof}

\section{Acknowledgements}

This work was funded in part by NSF grant DMS-1001370.

\end{document}